\documentclass[a4paper,12pt,titlepage]{article}
\usepackage[left=1in,right=1in,top=.8in,bottom=.8in,margin=1in]{geometry}
\usepackage{amssymb}
\usepackage{amsmath}
\usepackage{mathrsfs}
\usepackage{amsfonts}
\usepackage{setspace}
\usepackage{graphicx}

\newtheorem{theorem}{Theorem} 
\newtheorem{proof}{Proof} 
\newtheorem{lemma}{Lemma} 
 
\newtheorem{definition}{Definition}

\begin{document}

\title{%
	Descent of varieties using a Groebner basis argument}
\author{
	Deepak Kamlesh}
\date{} 
\maketitle

\pagenumbering{roman} 
\setcounter{page}{1} 
\tableofcontents
\newpage

\section{Introduction}
\pagenumbering{arabic}
In this article we prove a descent result for affine/projective varieties defined over an algebraically closed field.
The idea is to work with the reduced Groebner basis of the ideal where the variety vanishes and study it's behaviour under group action coming from subgroups of the automorphism group of the base field.

I should mention here that this idea came to me while working on my Masters Thesis Project on a 'p-adic analogue of Narasimhan-Seshadri theorem' under the guidance of Prof. Christopher Deninger at the University of Munster.

\section{What is a Groebner basis?}

We start by introducing the notion of a Groebner basis\footnote{What is a Groebner basis, Bernd Sturmfels, Notices of the AMS}.

Let $k$ be a field and $k[x_{1},\ldots,x_{n}]$ be the polynomial ring in $n$ variables over the field $k$.

\begin{definition}
A \textit{term order} on $k[x_{1},\ldots,x_{n}]$ is a total order $\prec$ on the set of all monomials $x^{a}=x_{1}^{a_{1}}\ldots x_{n}^{a_{n}}$ which has the following two properties:

\begin{enumerate}

\item It is multiplicative i.e. $x^{a} \prec x^{b}$ implies $x^{a+c} \prec x^{b+c}$ for all $a, b, c \in \mathbb{N}^{n}$

\item The constant monomial is the smallest i.e. $1 \prec x^a$ for all $a \in \mathbb{N} \backslash \{0\}$

An example of a term order (for n = 2) is the degree lexicographic order $1 \prec x_{1} \prec x_{2} \prec x_{1}^{2} \prec x_{1}x_{2} \prec x_{2}^{2} \prec x_{1}^{3} \ldots$

\end{enumerate}
\end{definition}

If we fix a term order $\prec$, then every polynomial $f$ has a unique \textit{initial term} $in_{\prec}(f)$. This is the largest monomial with non-zero coefficient in the expansion of $f$. The terms of $f$ are written in $\prec$ decreasing order. 

Suppose now that $I$ is an ideal $k[x_{1},\ldots,x_{n}]$. Then its \textit{initial ideal} $in_{\prec}(I)$ is the ideal generated by the initial terms of all the polynomials in $I$:
\begin{center}
$in_{\prec}(I) = <in_{\prec}(f) : f \in I>$
\end{center}

\begin{definition}
A finite generating set $G$ of $I$ is a $\textit{Groebner basis}$ with respect to the term order $\prec$ if the initial terms of the elements in $G$ suffice to generate the initial ideal:
\begin{center}
$in_{\prec}(I) = <in_(g):g \in G>$.
\end{center}
\end{definition}

Note that there is no minimality requirement for being a Groebner Basis. If $G$ is a Groebner basis for $I$, then any finite subset of $I$ that contains $G$ is also a Groebner basis. To remedy this non-minimality we have the notion of a \textit{reduced Groebner basis}.

\begin{definition}
A Groebner basis $G$ is said to be a \textit{reduced Groebner Basis} if 
\begin{enumerate}
\item For each $g \in G$, the coefficient of $in_{\prec}(g)|g \in G$ is $1$.
\item The set $\{in_{\prec}(g) | g \in G\}$ minimally generates $in_{\prec}(I)$.
\item No trailing term of any $g \in G$ lies in $in_{\prec}(I)$.
\end{enumerate}
\end{definition}

With this definition, we have the following well known theorem-
\begin{theorem}
If the term order $\prec$ is fixed then every ideal $I \in k[x_{1},\ldots,x_{n}]$ has a unique reduced Groebner basis. 
\end{theorem}

\section{Invariance of Groebner basis under isomorphisms}
Let $K / k$ be a field extension and let $\sigma \in Aut_{k}(K)$. Then, $\sigma$ induces an isomorphism of the polynomial ring $k[x_{1},\ldots,x_{n}]$ by acting on the coefficients. In particular $\sigma$ acts on the ideals of the polynomial ring.

We fix a term order $\prec$ on $k[x_{1},\ldots,x_{n}]$. Note that the isomorphism induced by $\sigma$ on the polynomial ring preserves ordering of monomials. 

Let $I$ be an ideal of $k[x_{1},\ldots,x_{n}]$. Let $G = \{f_{1},\ldots,f_{r}\}$ be a generating set for $I$. Let $^{\sigma}I$ be the conjugate ideal and let $^{\sigma}G = \{^{\sigma}f_{1},\ldots,^{\sigma}f_{r}\}$. Note that 
$^{\sigma}G$ is a generating set for $^{\sigma}I$. Then, we have the following lemma-

\begin{lemma}
let $G = \{f_{1},\ldots,f_{r}\}$ be the reduced Groebner basis of $I$ with respect to $\prec$. Then, $^{\sigma}G$ is the reduced Groebner basis for $^{\sigma}I$ with respect to $\prec$.
\end{lemma}

\begin{proof}
We check the three conditions for a generating set to be a reduced Groebner basis step by step.
\begin{enumerate}
\item First condition holds because a field automorphism takes identity to identity. 
\item For the second condition note that the set $\{in_{\prec}(g) | g \in ^{\sigma}G = \{in_{\prec}^{\sigma}g | g \in G\}$ and so generates $in_{\prec}(^{\sigma}I) = in_{\prec}^{\sigma}(I)$.
Further, if a proper subset of $^{\sigma}G$ generates $^{\sigma}I$ then the proper subset of $G$ obtained by applying $\sigma^{-1}$ to these generators gives a generating set for $in_{\prec}(I)$ contradicting that $G$ is a reduced Groebner basis.  
\item This condition follows by the same argument as before. If a trailing term of $g \in ^{\sigma}G$ lies in $in_{\prec}(^{\sigma}I)$ then applying $\sigma^{-1}$ to $g$ we get a trailing term of $^{\sigma^{-1}}g$ in $in_{\prec}(I)$ contradicting that $G$ is a reduced Groebner basis.     
\end{enumerate}
\end{proof}

Next, we will apply this result to affine varieties.

\section{Descent of affine varieties}
Setup:
Let $L/K$ be a field extension, let $L$ be algebraically closed and $A \subset Aut_{K}(L)$ be a subgroup of the automorphism group of $L$ over $K$ such that $K' = L^{A} := \{a \in L | \sigma(a) = a \; \forall \sigma \in L\}$ is the fixed field .

\begin{theorem}
Let $V$ be an affine variety defined over $L$ i.e. $V$ is the vanishing set of a radical ideal $I = (f_{1},\ldots,f_{r})$ generated by finitely many polynomials in $L[X_{1},\ldots,X_{k}]$. Then, $A$ acts on $V$ by acting on the coordinates. Suppose that $V$ is invariant under this action. Then, $V$ is defined over $K'$. 
\end{theorem}

\begin{proof}
By the correspondence between affine varieties and radical ideals we get that $^{\sigma}I = I$ for all $\sigma \in A$. WLOG we may assume that $G=(f_{1},\ldots,f_{r})$ is a reduced Groebner basis of $I$. (For concreteness assume that we are working with the lexicographic ordering). 

Then, from the lemma and the uniqueness of reduced Groebner basis we get that $^{\sigma}G = G$ for all $\sigma \in A$. 
It may be possible that ${\sigma}$ rearranges the elements of $G$. However, since no two elements of the reduced Groebner basis can have the same leading term this is not possible. So $^{\sigma}G = G$ as an ordered set i.e. $^{\sigma}f_{i} = f_{i}$ for all $i = 1, \ldots, r$. 

This means that all the coefficients of all the polynomials $f_{i}$ are invariant under the action of $A$ and hence actually belong to the fixed field $K'$. 
Therefore, $V$ is actually defined over $K'$. 
\end{proof}

Remarks-
\begin{itemize}
\item This is a quite elementary proof of descent as compared to the classical proofs of Galois Descent\footnote{J. S. Milne, Algebraic Geometry Chapter 16 - Descent Theory}.
Further, our result extends the classical result since we don't need to work with the full Galois group, rather subgroups of the automorphism group work as well. The only thing that needs to be taken into account is the fixed field.

\item The same proof works for descent of projective varieties as well by making use of the correspondence between homogenous radical ideals and projective varieties. Though it is yet to be seen whether this argument can be extended to arbitrary varieties.
\end{itemize} 

The notion of Groebner basis has been widely generalized and extended over the years for e.g. to polynomial over principal ideal rings and even to some classes of non-commutative rings. 

It is possible to extend our lemma to these general situations (though notions of a reduced Groebner basis and uniqueness doesn't always hold). However, their implications to descent problems and algebraic geometric interpretations are yet to be understood properly. So, this will be included in a later update of the article along with some results for arbitrary varieties.


\newpage
\section{Bibliography}
\begin{thebibliography}{9}
\bibitem{GB} Bernd Sturmfels, What is a Groebner basis, Notices of the AMS
\bibitem{C} Keith Conrad, Galois Descent \\
www.math.uconn.edu/\textasciitilde kconrad/blurbs/galoistheory/galoisdescent.pdf
\bibitem{Milne} J. S. Milne, Descent Theory, www.jmilne.org/math/CourseNotes/AG16.pdf
\bibitem{BW} Becker and Weispfenning, Groebner Basis, Springer publications
\end{thebibliography}
\end{document}